\documentclass[11pt,oneside,reqno]{amsart}

\usepackage{amsmath,amsthm} 
\usepackage{enumerate}
\usepackage{xfrac}          
\numberwithin{equation}{section}
\theoremstyle{definition}
\newtheorem{Definition}{Definition}[section]
\newtheorem{Example}[Definition]{Example}
\newtheorem{Remark}[Definition]{Remark}
\newtheorem{Notation}[Definition]{Notation}

\theoremstyle{plain}
\newtheorem{Theorem}[Definition]{Theorem}

\newtheorem{Proposition}[Definition]{Proposition}
\newtheorem{Corollary}[Definition]{Corollary}
\newtheorem{Lemma}[Definition]{Lemma}

\usepackage{amssymb}        
\usepackage{mathrsfs}       
\usepackage{eucal}          
\usepackage{stmaryrd}       

\usepackage{geometry}

\usepackage{hyperref}

\usepackage[usenames,dvipsnames]{xcolor}
\usepackage{tikz}
\usetikzlibrary{arrows, matrix}
\usepackage{tikz-cd}
\usepackage{subfig}         
\usepackage{arydshln}       

\newcommand{\al}{\alpha}

\newcommand{\ga}{\gamma}
\newcommand{\Ga}{\Gamma}

\newcommand{\ep}{\varepsilon}

\newcommand{\la}{\lambda}
\newcommand{\La}{\Lambda}
\newcommand{\si}{\sigma}

\newcommand{\Z}{\mathbb{Z}}

\newcommand{\C}{\mathbb{C}}
\newcommand{\K}{\Bbbk}

\newcommand{\Fgl}{\mathfrak{gl}}

\newcommand{\Fb}{\mathfrak{b}}
\newcommand{\Fg}{\mathfrak{g}}

\newcommand{\Fk}{\mathfrak{k}}

\newcommand{\Fm}{\mathfrak{m}}
\newcommand{\Fn}{\mathfrak{n}}

\newcommand{\Ft}{\mathfrak{t}}

\newcommand{\CC}{\mathcal{C}}
\newcommand{\CD}{\mathcal{D}}
\newcommand{\CF}{\mathcal{F}}
\newcommand{\CH}{\mathcal{H}}

\newcommand{\CK}{\mathcal{K}}

\newcommand{\CS}{\mathcal{S}}

\newcommand{\op}{\operatorname}

\DeclareMathOperator{\Ann}{Ann}
\DeclareMathOperator{\Aut}{Aut}
\DeclareMathOperator{\codim}{codim}

\DeclareMathOperator{\End}{End}

\DeclareMathOperator{\Frac}{Frac}

\DeclareMathOperator{\Hom}{Hom}
\DeclareMathOperator{\Id}{Id}

\DeclareMathOperator{\MaxSpec}{MaxSpec}
\DeclareMathOperator{\cfs}{MaxSpec^{cf}}

\DeclareMathOperator{\Stab}{Stab}

\DeclareMathOperator{\GL}{GL}

\renewcommand{\hat}{\widehat}
\renewcommand{\tilde}{\widetilde}

\newcommand{\tri}{\triangleright}
\newcommand{\defby}{\stackrel{\text{\tiny def}}{=}}

\title{Harish-Chandra modules over Hopf Galois orders}
\author{Jonas T. Hartwig}
\date{May 1, 2021}

\address{Department of Mathematics, Iowa State University, Ames, IA-50011, USA}
\email{jth@iastate.edu}
\urladdr{http://jthartwig.net}

\begin{document}
\begin{abstract}
The theory of Galois orders was introduced by Futorny and Ovsienko \cite{FutOvs2010}. We introduce the notion of $\mathcal{H}$-Galois $\La$-orders. These are certain noncommutative orders $F$ in a smash product of the fraction field of a noetherian integral domain $\La$ by a Hopf algebra $\CH$ (or, more generally, by a coideal subalgebra of a Hopf algebra). They are generalizations of Webster's principal flag orders \cite{Web2019}. Examples include Cherednik algebras, as well as examples from Hopf Galois theory. We also define spherical Galois orders, which are the corresponding generalizations of principal Galois orders introduced by the author \cite{Har2020}.

The main results are
(1) for every maximal ideal $\mathfrak{m}$ of $\La$ of finite codimension, there exists a simple Harish-Chandra $F$-module in the fiber of $\mathfrak{m}$;
(2) for every character of $\La$ we construct a canonical simple Harish-Chandra module as a subquotient of the module of local distributions;
(3) if a certain stabilizer coalgebra is finite-dimensional, then the corresponding fiber of simple Harish-Chandra modules is finite;
(4) centralizers of symmetrizing idempotents are spherical Galois orders and every spherical Galois order appears that way.
\end{abstract}

\maketitle

\section{Introduction}

In \cite{DroFutOvs1994}, the category of Harish-Chandra modules over an associative algebra $A$ with respect to a Harish-Chandra subalgebra $\Gamma$ was introduced, and some tools to study it were provided. The objects are $A$-modules which are sums of finite-dimensional $\Gamma$-modules. 
Examples include Harish-Chandra modules over $U(\Fg)$ with respect to $U(\Fk)$ (where $\Fk$ is a reductive or nilpotent subalgebra of a finite-dimensional Lie algebra $\Fg$ over a field of characteristic zero), which motivates the name.
Another example is the category of Gelfand-Tsetlin modules over $U(\Fgl_n)$. These are Harish-Chandra modules with respect the Gelfand-Tsetlin subalgebra $\Ga$ of $U(\Fgl_n)$, generated by $\big\{Z\big(U(\Fgl_k)\big)\big\}_{k=1}^n$.

In \cite{FutOvs2010}\cite{FutOvs2014}, a class of algebras called \emph{Galois orders} (or \emph{Galois algebras}) was introduced and investigated. These are subalgebras of the algebra of $G$-invariants in a skew monoid algebra $L\rtimes\mathcal{M}$ with respect to the action of a finite group $G$, and come with a commutative Harish-Chandra subalgebra $\Ga$.
One of the main results of \cite{FutOvs2014} is a sufficient condition for there to be only finitely many isomorphism classes of simple Harish-Chandra modules supported on a given maximal ideal. The main tool was the framework from \cite{DroFutOvs1994}.
So called \emph{principal Galois orders} were introduced in \cite{Har2020}. These are certain Galois orders which naturally act on their subalgebra $\Ga$. 
This action was pointed out in \cite{Vis2018}, in the case of $U(\mathfrak{gl}_n)$.
It was shown in \cite{FutMolOvs2010} that finite W-algebras of type A are examples of Galois orders and in \cite{Har2020} that they are principal Galois orders.
In \cite[Sec.~1.3]{Wee2019} it was shown that orthogonal Gelfand-Tsetlin (OGZ) algebras from \cite{Maz1999} are in fact isomorphic to finite W-algebras of type A.
Webster \cite{Web2019} showed that the theory of principal Galois orders can be clarified by regarding these algebras as centralizer subalgebras of Galois orders with trivial $G$ he called \emph{principal flag orders}. In \cite{Web2019} it was shown that Coulomb branches in the sense of Braverman-Finkelberg-Nakajima \cite{BraFinNak2018} are examples of principal Galois orders. In \cite{LePWeb2019} it was shown that the spherical subalgebra of rational Cherednik algebras of imprimitive complex reflection groups $G(\ell,p,n)$ are also principal Galois orders. 

In this paper, we propose a generalization of principal flag orders, in which the skew monoid algebra $L\rtimes\mathcal{M}$ is replaced by a smash product $L\#\CH$, where $\CH$ is a coideal subalgebra of a Hopf algebra.
To make the conditions on these smash products precise, in Section \ref{sec:settings} we introduce the notion of a \emph{(Hopf Galois order) setting} $(\CH,\La)$, and provide numerous examples of such settings in Examples \ref{ex:Hopf-Galois}--\ref{ex:quantum-Borel}. In particular, Example \ref{ex:Hopf-Galois} implies that Hopf Galois extensions $\La^{\op{co}\CH}\subset\La$ give rise to settings $(\CH^\circ, \La)$.

In Section \ref{sec:Hopf-Galois-orders}, we define the algebras of interest. These are certain subalgebras $F$ of the smash products $L\#\CH$. We call them \emph{Hopf Galois orders} (or \emph{$\CH$-Galois $\La$-orders} when referencing the setting). 
We follow \cite{Web2019} and regard the centralizer (or ``spherical'') subalgebras as a secondary object, and discuss these later in Section \ref{sec:spherical}. 
See Figure \ref{fig:figure1} for the relationship between the different classes of algebras.
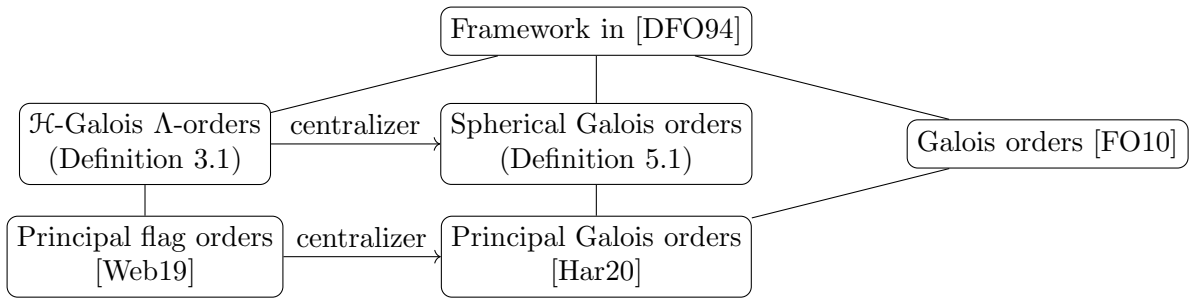
\begin{figure}
\begin{center}
\begin{tikzpicture}[every node/.style = {shape=rectangle, rounded corners, draw, align=center}]
\node (DFO) at (0,3) {Framework in \cite{DroFutOvs1994}};
\node (HGO) at (-6,1.5) {$\CH$-Galois $\La$-orders\\ (Definition \ref{def:Galois-order})};
\node (PFO) at (-6,0){Principal flag orders\\  \cite{Web2019}};
\node (SGO) at (0,1.5) {Spherical Galois orders\\ (Definition \ref{def:spherical-Galois-order})};
\node (PGO) at (0,0) {Principal Galois orders\\ \cite{Har2020}};
\node (GO) at (6,1.5) {Galois orders \cite{FutOvs2010}};
\draw (DFO) -- (HGO) -- (PFO);
\draw (DFO) -- (SGO) -- (PGO);
\draw (DFO) -- (GO) -- (PGO);
\draw[->] (HGO) -- node[draw=none, auto] {centralizer} (SGO);
\draw[->] (PFO) -- node[draw=none, auto] {centralizer} (PGO);
\end{tikzpicture}
\end{center}
\caption{Classes of algebras. Up is more general.}
\label{fig:figure1}
\end{figure}
We give several examples, including rational Cherednik algebras attached to an arbitrary subgroup of $\GL(V)$, using the Dunkl-Opdam representation. 
(Thus the analogy between $U(\Fg)$ and the sperical subalgebra of rational Cherednik algebras alluded to in \cite[Analogy 3]{Rou2005} can be made precise (when $\Fg=\Fgl_n$) by regarding both things as centralizer subalgebras of Hopf Galois orders.)
We also prove some lemmas, in particular that $\La$ is maximal commutative in any Hopf Galois order $F$ (Lemma \ref{lem:maxcomm}), and that the short exact sequence of left $\La$-modules $0\to\La\to F\to F/\La\to 0$ splits (Lemma \ref{lem:direct-sum}). The latter result is trivial to prove, but we believe is of some importance and have not seen it stated before. We end Section \ref{sec:Hopf-Galois-orders} by proving that Hopf Galois orders satisfy the Futorny-Ovsienko property (Theorem \ref{thm:FO}).

In Section \ref{sec:HC-modules} we prove the main results about Harish-Chandra modules over Hopf Galois orders, which can be summarized as follows. For a maximal ideal $\Fm\subset\La$ of finite codimension, let $\op{Irr}(F,\Fm)$ be the set of isomorphism classes of simple Harish-Chandra $F$-modules (with respect to $\La$) containing a nonzero generalized weight vector of weight $\Fm$.
In Theorem \ref{thm:nonempty}, we show that $\op{Irr}(F^{\mathrm{op}},\Fm)$ is non-empty. (We use the opposite algebra rather than talking about right modules in some of these theorems.) To the best of our knowledge such an existence theorem is new in this generality, even in the case of principal flag orders. The proof is not constructive but is based on the split short exact sequence mentioned above.

Theorem \ref{thm:canonical} concerns the case when $\Fm$ has codimension one. The historical context is as follows.
So called \emph{derivative tableaux} were introduced in \cite{FutGraRam2016} as basis vectors for singular Gelfand-Tsetlin modules over $U(\Fgl_n)$.
Vishnyakova interpreted these as local distributions \cite{Vis2018}.
In \cite{RamZad2018}, a simple Gelfand-Tsetlin module over $U(\Fgl_n)$ was constructed for each character $\la$ of the Gelfand-Tsetlin subalgebra in a uniform fashion.
The closely related notion of \emph{canonical simple Gelfand-Tsetlin module} was given in \cite{EarMazVis2020} in the more general setting of OGZ algebras. This was further generalized to principal Galois orders in \cite{Har2020} and principal flag orders in \cite{Web2019}. 
Building on these ideas, in Theorem \ref{thm:canonical} we construct, for any $\Fm$ of codimension one, a simple Harish-Chandra module in $\op{Irr}(F^{\mathrm{op}},\Fm)$ as a subquotient of the Harish-Chandra module of local distributions.

It was shown in \cite{SilWeb2020} that any simple Gelfand-Tsetlin module over $U(\Fgl_n)$ is a canonical module. It would be interesting to know the extent to which such a result can be generalized.

Our third result about Harish-Chandra modules over Hopf Galois orders provides a sufficient condition for the set $\op{Irr}(F,\Fm)$ to be finite (Theorem \ref{thm:main}). In the special case of principal flag orders, in \cite{Web2019} it is shown that $\op{Irr}(F,\Fm)$ is finite if a certain stabilizer subgroup is finite. The analogous statement for Galois orders was proved in \cite{FutOvs2014}. To generalize this result we need a Hopf algebra analog of the stabilizer subgroup. We discuss this somewhat technical topic in Section \ref{sec:stabilizer}. The power of Theorem \ref{thm:main} is that the sufficient condition only involves the setting $(\CH,\La)$ and the cofinite maximal ideal $\Fm$, but the conclusion applies equally to all Hopf Galois orders in that setting. On the other hand, one downside of this is that the result does not apply to rational Cherednik algebras in the rational-differential setting because there are Hopf Galois orders in the same setting for which the conclusion fails (see Example \ref{ex:infinite} for a prototypical case). In Theorem \ref{thm:finite-weight-spaces-general} we show that, under a weaker assumption involving the coradical filtration, the generalized weight spaces are finite-dimensional.

Lastly, in Section \ref{sec:spherical}, we specialize to the case when $\CH$ contains a finite group $W$ which allows us to consider the centralizer algebra $eFe$ where $e\in\K W$ is the symmetrizing idempotent of $W$ (we assume the characteristic of $\K$ does not divide the order of $W$).
First we define the notion of a \emph{spherical Galois order}. These algebras are generalizations of principal Galois orders from \cite{Har2020}. Then, generalizing \cite{Web2019}, we show that spherical Galois orders are the same thing as centralizer algebras $eFe$ of Hopf Galois orders. We state some corollaries of this fact and comment on the Morita equivalence between $F$ and $eFe$.

\section*{Acknowledgements}
The author gratefully acknowledges support from Simons Collaboration Grant for Mathematicians, award number 637600.
The author thanks Mark Colarusso, Sam Evens, Erich Jauch, Joanna Meinel, Catharina Stroppel, Akaki Tikaradze, and Ben Webster for helpful comments and discussions.

\section{Settings}\label{sec:settings}
Unless otherwise stated, we work over an arbitrary ground field $\K$.

\begin{Definition} \label{def:setting}
A \emph{Hopf Galois order setting} is a pair $(\CH,\La)$ where
 $\CH$ is a left coideal subalgebra of a Hopf algebra $\tilde\CH$ with comultiplication $\Delta:\tilde\CH\to \tilde\CH\otimes\tilde\CH$, $x\mapsto x_{(1)}\otimes x_{(2)}$, counit $\ep:\tilde\CH\to\K$, and invertible antipode $S:\tilde\CH\to\tilde\CH^{\mathrm{op},\mathrm{cop}}$;
and $\La$ is a left $\tilde\CH$-module algebra such that
\begin{enumerate}[{\rm (i)}]
\item $\La$ is a noetherian integral domain,
\item the action of $\tilde\CH$ on $\La$ extends to an action on the field of fractions $L=\Frac(\La)$,
\item $\La$ is faithful as a left module over the smash product $\La\#\CH$.
\end{enumerate}
\end{Definition}

\begin{Remark} \phantom{X}
\label{sec:settings-remark}
\begin{enumerate}
\item We regard $\tilde\CH$ as part of the data which specifies $\CH$.
\item Condition (ii) holds for example if the coradical of $\tilde\CH$ is finite-dimensional or cocommutative \cite{SkrVan2006}, or if $\tilde\CH$ acts locally finitely on $\La$ \cite{Skr2020}.
\item Condition (iii) holds iff $L$ is faithful as a left module over $L\#\CH$, and is well-known to be related to Hopf Galois theory, see Example \ref{ex:Hopf-Galois}.
\end{enumerate}
\end{Remark}

\begin{Notation} \phantom{X}
\begin{enumerate}
\item The field of fractions of $\La$ will be denoted by $L$.
\item The left action of $\tilde\CH$ on $\La$ will be denoted by 
$\tri:\tilde\CH\otimes \La \to \La$,\; $x\otimes f\mapsto x\tri f$. Thus,
\begin{align*}
(xy)\tri a &= x\tri(y\tri a),& 1_\CH\tri a &= a,\\
x\tri(ab)&=(x_{(1)}\tri a)\cdot (x_{(2)}\tri b),& x\tri 1_\La&=\ep(x)1_\La,
\end{align*}
for all $x,y\in\tilde\CH$ and $a,b\in \La$. The same symbol $\tri$ is used for the action of $\tilde\CH$ on $L$.
\item If $A$ is a left comodule algebra over a Hopf algebra $\tilde\CH$, and $\CH$ is a left coideal subalgebra of $\tilde\CH$, the \emph{smash product} $A\#\CH$ is
the unique algebra satisfying
\begin{enumerate}[{\rm (\#1)}]
\item $A$ and $\CH$ are subalgebras of $A\#\CH$;
\item $A\otimes\CH\to A\#\CH$, $a\otimes x\mapsto ax$ is a vector space isomorphism;
\item the following cross relation holds for all $x\in\CH$ and $a\in A$:
\begin{equation}\label{eq:smash}
xa = (x_{(1)}\tri a)\cdot x_{(2)}.
\end{equation}
\end{enumerate}
\item When $W$ is a group acting by automorphisms on a Hopf algebra $\tilde\CH$ we put $\tilde\CH\rtimes W\defby \tilde\CH\#\K W$, which is naturally a Hopf algebra. Similarly, if $W$ preserves a left coideal subalgebra $\CH$ of $\tilde\CH$, then $\CH\rtimes W\defby\CH\#\K W$ is a left coideal subalgebra of $\tilde\CH\rtimes W$.
\item For $X\in L\#\tilde\CH$ we let $\hat X$ denote the corresponding linear endomorphism of $L$. Explicitly,
\begin{align}
\hat f(g)&=fg, \quad \forall f,g\in L,\\
\hat x(f)&=x\tri f, \quad \forall x\in\tilde\CH,\,f\in L.
\end{align}
In particular, for any $X\in L\#\tilde\CH$ and $f\in L$, we have
 $\hat X(f) = \hat{X}\hat f(1_L)$.
\end{enumerate}
\end{Notation}

\begin{Example}\label{ex:Hopf-Galois}
Let $\CH$ be a Hopf algebra and $L$ a right $\CH$-comodule algebra such that $L$ is a field. 
Let $\CH^\circ$ be the finite dual and regard $L$ as a left $\CH^\circ$-module algebra.
Suppose that the Galois map $\beta:L\otimes_{L^{\op{co}\CH}} L \to L \otimes \CH,\, a\otimes b\mapsto ab_{(0)}\otimes b_{(1)}$ is surjective.
Then $L$ is a faithful $L\#\CH^\circ$-module. Consequently, if $\La\subset L$ a noetherian $\CH^\circ$-module subalgebra with fraction field $L$, then 
$(\CH^\circ,\La)$ is a Hopf Galois order setting.
\begin{proof}
Properties (i) and (ii) hold by construction, so it remains to prove that (iii) holds.
In the special case when $\CH$ is a finite-dimensional Hopf algebra, $\beta$ is surjective iff $L$ is a faithful $L\#\CH^\circ$-module (\cite[Thm.~8.3.1 and Thm.~8.3.7]{Mon1993}).
The general case must be well-known but we were not able to find the precise statement in the literature and therefore provide the details.
Let $X=\sum_{i=1}^n a_ix_i\in L\#\CH^\circ$ where $\{x_i\}_{i=1}^n$ is linearly independent over $\K$ and $a_i\in L$. For $h\in\CH$, let $h^{[1]}\otimes h^{[2]}\in\beta^{-1}(\{1\otimes h\})$. Since $\beta$ is a map of right $\CH$-comodules one checks that $h^{[1]}{h^{[2]}}_{(0)}\otimes {h^{[2]}}_{(1)}=1\otimes h$ (see eg. \cite{Sch1990}). Thus $h^{[1]}\hat X(h^{[2]}) = \sum_{i=1}^n a_i h^{[1]} (x_i \tri h^{[2]})=\sum_{i=1}^n a_i h^{[1]} {h^{[2]}}_{(0)} x_i({h^{[2]}}_{(1)})= \sum_i a_i x_i(h)$. Let $h_i\in\CH$ with $x_i(h_j)=\delta_{ij}$. Thus, if $\hat X(b)=0$ for all $b\in\La$, then $0=h_i^{[1]}\hat X(h_i^{[2]})=a_i$ for all $i$.
\end{proof}
\end{Example}

\begin{Example}\label{ex:sub}
If $(\CH,\La)$ is a Hopf Galois order setting, and $\CH'$ is a left coideal subalgebra of $\tilde\CH$ contained in $\CH$, then $(\CH',\La)$ is a Hopf Galois order setting.
\end{Example}

\begin{Example}
Let $(\CH,\La)$ be any Hopf Galois order setting. Let $\La[t]=\La\otimes\K[t]$ where $t$ is an indeterminate. Let $\CH[\partial]=\CH\otimes\K[\partial]$ where $\partial$ is a primitive generator acting as $\Id_\La\otimes\frac{d}{dt}$ on $\La[t]$. If $\op{char}\K=0$ then $(\CH[\partial],\La[t])$ is a Hopf Galois order setting. Indeed, the relation $\partial^n t - t\partial^n = n\partial^{n-1}$ shows that any nonzero $(\K[t],\K[t])$-subbimodule of $\La[t]\#\CH[\partial]$ has nonzero intersection with $\La[t]\#\CH$ and $\La[t]\#\CH\cap \Ann_{\La[t]\#\CH[\partial]}(\La[t])=\Ann_{\La\#\CH}(\La)[t]=0$. Similarly, if $\op{char}\K=p$ then $(\CH[\partial]/(\partial^p),\La[t])$ is a Hopf Galois order setting.

More generally, if $\partial$ acts by derivations on $\tilde\CH$, perserving $\CH$, and on $\La$ in a compatible way (i.e. $\partial\tri(x\tri a)-x\tri(\partial\tri a)=(\partial\tri x)\tri a$ for all $x\in\tilde\CH$ and $a\in\La$) then $(\CH\#\K[\partial],\La[t])$ is a Hopf Galois order setting when $\op{char}\K=0$, and $(\CH\#\K[\partial]/(\partial^p),\La[t])$ is a Hopf Galois order setting when $\op{char}\K=p$ and $\partial^p$ acts as zero on $\tilde\CH$ and $\La$.
\end{Example}

\begin{Example}
\label{ex:example1}
Let $\La$ be a noetherian integral domain over $\K$.
Let $\Fg$ be a Lie algebra acting by derivations on $\La$, and let $U(\Fg)$ be the universal enveloping algebra of $\Fg$. Suppose that
 $(\La\Fg)\cap\Ann_{\La\# U(\Fg)}(\La)=0$.
Then $(U(\Fg),\La)$ is a Hopf Galois order setting. Furthermore, if $W$ is a group acting by algebra automorphisms on $\La\# U(\Fg)$ preserving $\La$ and $U(\Fg)$, such that $W$ is acting faithfully on $\La$, and by Hopf algebra automorphisms on $U(\Fg)$, then $( U(\Fg)\rtimes W, \La)$ is a Hopf Galois order setting.

\begin{proof}
Since $U(\Fg)\rtimes W$ is a cocommutative Hopf algebra, condition (ii) is automatic, by Remark \ref{sec:settings-remark}(2). 
It remains to show that $\La$ faithful as a left module over $\La\# U(\Fg)\rtimes W$.

\noindent \underline{Step 1}: By \cite{BerMon1986}, every nonzero $(\La,\La)$-subbimodule of $\La\# U(\Fg)$ has nonzero intersection with
$\La$. (Since $\La$ is commutative, the extended centroid $C$ is just the field of fractions $L$ of $\La$. We have
 $L\Fg\hookrightarrow\End(L)$ which means that the action of $\Fg$ on $\La$ is $L$-outer.)

\noindent\underline{Step 2}: We show that every nonzero $(\La,\La)$-subbimodule $M$ of $\La\# U(\Fg)\rtimes W$ has nonzero intersection with
$\La W$. It suffices to show this for $M=\La X\La$ where $0\neq X\in \La\# U(\Fg)\rtimes W$ is arbitrary. 
Write $X=X_1 w_1+X_2w_2+\cdots+X_n w_n$ where $0\neq X_i\in\La\# U(\Fg)$.
We proceed by induction on $m=m(X)\defby \#\{i\in\{1,2,\ldots,n\}\mid X_i\notin\La\}$.
If $m=0$ then $\La X\La \cap \La W=\La X\La\neq 0$. If $m>0$, without loss of generality $X_1\notin\La$.
By Step 1, $\La X_1\La \cap \La\neq 0$. Thus 
there exists $\sum_i a_i\otimes b_i\in\La\otimes\La$ such that 
$\tilde X \defby \sum_i a_i X w_1^{-1}(b_i) = 
\big(\sum_i a_i X_1 b_i\big)w_1 + \cdots$ 
has $w_1$-coefficient $\tilde{X}_1=\sum_i a_i X_1 b_i\in\La\setminus\{0\}$.
Moreover, if $X_j\in\La$ then the $w_j$-coefficient of $\tilde X$ is $\tilde{X}_j = 
\sum_i a_i X_j w_jw_1^{-1}(b_j)\in \La$. Thus  $m(\tilde X)\le m(X)-1$
because $\tilde{X}_1\in\La$ while if $X_j\in\La$ then $\tilde X_j\in\La$ too. Now $\La X\La \cap \La W\supset \La \tilde{X}\La \cap \La W\neq 0$ by induction.

\noindent\underline{Step 3}: Let $M=\Ann_{\La\# U(\Fg)\rtimes W}(\La)$ and suppose that $M\neq 0$. Since $M$ is an ideal, it is a $(\La,\La)$-subbimodule. By Step 2, there exists some nonzero $X\in M\cap \La W$. However this contradicts Dedekind's Independence Theorem (which states that $\La W$ acts faithfully on $\La$ when $W$ does).
\end{proof}
\end{Example}

\begin{Example}
\label{ex:example2} 
(Special case of Example \ref{ex:example1}.)
Take $\K=\C$.
Let $G$ be a connected affine algebraic group, $H$ a finite subgroup of $G$ and put $X=H\backslash G$ the set of right cosets of $H$ in $G$. Let $\C[X]$ be the algebra of regular functions on $X$. $G$ acts on $X$ by right multiplication hence from the left on $\C[X]$ by automorphisms. Differentiating this action gives an action of the Lie algebra $\Fg$ of $G$ on $\C[X]$ by derivations. Let $W$ be a group acting by algebra automorphisms on $\C[X]\# U(\Fg)$ preserving $\C[X]$ and $U(\Fg)$, such that $W$ is acting faithfully on $\C[X]$, and by Hopf algebra automorphisms on $U(\Fg)$.
Then $\big(U(\Fg)\rtimes W, \C[X]\big)$ is a Hopf Galois order setting.

\begin{proof}
Since $X$ is a smooth irreducible affine variety, $\C[X]$ is a noetherian integrally closed domain. By Example \ref{ex:example1} it now suffices to show that $\C[X]\Fg \cap \Ann_{\C[X]\# U(\Fg)\rtimes W}(\C[X])=0$. Since $H$ is finite, the orbit map $\alpha:G\to X$ is an immersion at every point. That is, for all $p\in X$, $\mathrm{d}\alpha_p:\Fg\to T_p(X)$, $v\mapsto v_p$ is injective. Let $\{v_i\}$ be a basis for $\Fg$ regarded as left invariant vector fields on $G$. Suppose that $\sum_i f_i v_i=0$ for some $f_i\in\C[X]$ not all of which are identically zero. Then at $p$ we have $\sum_i f_i(p)(v_i)_p=0$ which contradicts that $\{(v_i)_p\}$ is a basis for the tangent space $T_p(X)$. 
\end{proof}
\end{Example}

\begin{Example}[Special case of Example \ref{ex:example2}]
Let $G=\GL(n,\C)$ and $\C[G]=\C[x_{ij}\mid 1\le i,j\le n][\det^{-1}]$ the algebra of regular functions. Let $H=S_n$ be the subgroup of all permutation matrices, acting by left multiplication on $G$. Let $X=S_n\backslash G$ be the set of orbits (right cosets of $S_n$ in $G$). Then $S_n$ acts from the right on $\C[G]$ by automorphisms. Explicitly, $x_{ij}.\sigma=x_{\si^{-1}(i)j}$ for $\si\in S_n$, and $\C[X]\cong\C[G]^{S_n}$. On the other hand, $G$ acts from the right on $G$ hence from the left on $\C[G]$. Since these actions commute we get an induced left action of $G$ on $\C[X]$. Differentiating this gives a Lie algebra homomorphism
$\Fgl_n\to \op{Der}(\C[X])$ given by $E_{ij}\mapsto \sum_{k=1}^n x_{ki}\frac{\partial}{\partial x_{kj}}$. Finally, let $W=S_n$, acting from the right on $G$ by conjugation, hence on $\C[X]$ from the left. Explicitly,  $\sigma.x_{ij}=x_{\si(i)\si(j)}$ for $\si\in S_n$. This induces a left action of $S_n$ on $U(\Fgl_n)$ by $\si.E_{ij}=E_{\si(i)\si(j)}$. Then $\big(U(\Fgl_n)\rtimes S_n,\, \C[S_n\backslash GL_n]\big)$ is a Hopf Galois order setting.
\end{Example}

\begin{Example}[Rational-differential setting, special case of Example \ref{ex:example2}]
\label{ex:rational-differential}
Let $V$ be a finite-dimensional complex vector space and let $G=V$ be the additive group of $V$ acting on itself by left translation. Let $W$ be a subgroup of $\GL(V)$ acting naturally on the symmetric algebra $S(V)$ and on $\C[V]$. Then $\big(S(V)\rtimes W, \C[V]\big)$ is a Hopf Galois order setting.
\end{Example}

\begin{Example}[Trigonometric-differential setting, special case of Example \ref{ex:example2}]
Let $G=T\cong (\C^\times)^n$ be a complex torus acting on itself by left multiplication and $\C[T]\cong\C[z_1^{\pm 1},z_2^{\pm 1},\ldots,z_n^{\pm 1}]$ be the algebra of regular functions on $T$. Let $\Ft\cong \C^n$ be the (abelian) Lie algebra of $T$, $U(\Ft)\cong \C[z_1\frac{\partial}{\partial z_1}, z_2\frac{\partial}{\partial z_2},\ldots, z_n\frac{\partial}{\partial z_n}]$ be the universal enveloping algebra. Let $W$ be a group acting by algebraic group automorphisms on $T$ and induced actions on $\Ft$, $U(\Ft)$ and $\C[T]$. Then $\big(U(\Ft)\rtimes W, \C[T]\big)$ is a Hopf Galois order setting.
\end{Example}

\begin{Example}[Quantum Borel setting]\label{ex:quantum-Borel}
Suppose $q$ is an element of infinite order in the multiplicative group of $\K$.
Let $\tilde\CH$ be the algebra with generators $\{E,K,K^{-1}\}$ and relations
\[ KEK^{-1}=qE,\quad KK^{-1}=K^{-1}K=1.\]
There exists a Hopf algebra structure on $\tilde\CH$ determined by
\[\Delta(K)=K\otimes K,\quad \ep(K)=1,\quad S(K)=K^{-1}\]
\[\Delta(E)=E\otimes 1 + K\otimes E,\quad \ep(E)=0,\quad S(E)=-K^{-1}E.\]
Let $\CH=\K[E]$ be the subalgebra of $\tilde\CH$ generated by $E$. It is a left coideal subalgebra of $\tilde\CH$.
Let $\La=\K[t]$. Define an action of $\tilde\CH$ on $\La$ by
\[ E\tri t=1,\quad K\tri t=q^{-1}t. \]
This is well-defined because 
\[(KEK^{-1}-qE)\tri t = K\tri(E\tri(qt))-q= qK\tri 1 -q =0.\]
Using the coproduct Leibniz rules
\[E\tri(fg)=(E\tri f)g+(K\tri f)(E\tri g)\]
\[K\tri(fg)=(K\tri f)(K\tri g)\]
it is easy to see that
\[ E\tri f(t) = \frac{f(t)-f(q^{-1}t)}{t-q^{-1}t}. \]
Thus $(t-q^{-1}t)E-(1-K)\in\Ann_{\La\#\tilde\CH}\La$ so $\La\#\tilde\CH$ does not act faithfully on $\La$.
However, the algebra $\La\#\CH$ does act faithfully on $\La$. To see this, first observe that the smash product $\La\#\CH$ is isomorphic to the algebra generated by
$t$ and $E$ subject to $Et=(E_{(1)}\tri t)E_{(2)}$ or, more explicitly,
\[Et=1+q^{-1}tE.\]
Thus $\La\#\CH$ is isomorphic to a quantum Weyl algebra.
Inductively, $E^nt - q^{-n} tE^n$ is a nonzero multiple of $E^{n-1}$ for any $n\ge 0$.
This implies that every nonzero $(\La,\La)$-subbimodule of $\La\#\CH$ intersects $\La$ nontrivially.
Since $\La$ acts faithfully on itself ($1_\La\in\La$), $\La\cap\Ann_{\La\#\CH}(\La)=0$. Since 
$\Ann_{\La\#\CH}(\La)$ is an ideal, it is a $(\La,\La)$-subbimodule, hence $\Ann_{\La\#\CH}(\La)=0$.
So $\La\#\CH$ acts faithfully on $\La$.

Note, though, that $\K(t)\#\K[E]\cong \K(t)\#\K[K]$ as $\K(t)$-rings, via $E\mapsto \frac{1}{t-q^{-1}t}(1-K)$.
Thus instead of working with the left coideal subalgebra $\K[E]$ one could work with the monoid bialgebra $\K[K]$. An advantage of this is that the theory of Galois orders as originally developed in \cite{FutOvs2010} applies. On the other hand, the coradical filtration on $U_q(\Fb)$ gives rise to a filtration on $\K[E]$ (such that $\K[E]_n=\K 1+\K E+\cdots+ \K E^n$). This structure is not ``naturally'' available on $\K[K]$ from the perspective of Hopf algebras.
\end{Example}

\section{Hopf Galois orders}\label{sec:Hopf-Galois-orders}

Let $(\CH,\La)$ be a Hopf Galois order setting.

\begin{Definition}
\label{def:Galois-order}
An \emph{$\CH$-Galois $\La$-order} is a subalgebra $F\subset L\#\CH$ satisfying
\begin{enumerate}[{\rm (i)}]
\item $\La\subset F$,
\item $LF=L\#\CH$,
\item $\hat X(\La)\subset \La$ for all $X\in F$.
\end{enumerate}
\end{Definition}

\begin{Remark}\phantom{X}
\begin{enumerate}
\item The antipode $S:\tilde\CH\to\tilde\CH^{\mathrm{cop,op}}$ can be extended to an algebra isomorphism
\begin{equation}
S:L\#\tilde\CH \to (L\#\tilde\CH^{\mathrm{cop}})^{\mathrm{op}}
\end{equation}
which is the identity map when restricted to $L$. This follows from the following identity in $L\#\tilde\CH$:
\begin{equation}\label{eq:aSx}
aS(x)=S(x_{(2)})(x_{(1)}\tri a),\quad\forall a\in L, x\in\tilde\CH.
\end{equation}
\end{enumerate}
\end{Remark}

\begin{Example}
The \emph{standard $\CH$-Galois $\La$-order} is defined by
\begin{equation}
\CF(\CH,\La) = \{X\in L\#\CH\mid \hat X(\La)\subset\La\}.
\end{equation}
It is easy to check that $L \CF(\La,\CH)=L\#\CH$, hence $\CF(\CH,\La)$ is an $\CH$-Galois $\La$-order.
\end{Example}

\begin{Example} \label{ex:sub-galois}
If $F$ is an $\CH$-Galois $\La$-order and $\CH'$ is a left coideal subalgebra of $\tilde\CH$ contained in $\CH$, then $F(\CH')\defby F\cap L\CH'$ is an $\CH'$-Galois $\La$-order.

\begin{proof}
As mentioned in Example \ref{ex:sub}, if $(\CH,\La)$ is a setting then so is $(\CH',\La)$. Put $F'=F(\CH')$. Since $1_{\tilde\CH}\in\CH'$ we have $\La\subset F'$. Since $F'\subset F$ we have $\hat X(\La)\subset\La$ for all $X\in F'$. It remains to show that $LF'=L\CH'$. It suffices to show $\CH'\subset LF'$. Let $x\in\CH'$ be arbitrary. Since $\CH'\subset \CH\subset LF$ there exists $a\in\La\setminus\{0\}$ such that $ax\in F$.  On the other hand $ax\in L\CH'$. Thus $ax\in F'$ by definition of $F'$. Therefore $x\in a^{-1}F'\subset LF'$. 
\end{proof}
\end{Example}

\begin{Example}
In the rational-differential setting (Example \ref{ex:rational-differential}), the standard $S(V)\rtimes W$-Galois $\C[V]$-order $\CF(S(V)\rtimes W,\C[V])$ is the nilHecke algebra of $(V,W)$.
\end{Example}

\begin{Example}
In the quantum Borel setting $(\C[E]\subset U_q(\Fb), \C[t])$ from Example \ref{ex:quantum-Borel}, let $p(t)\in\C[t]$ be any nonzero polynomial and define $F(p)$ to be the subalgebra of $\C[t]\#\C[E]$ generated by $t$ and $p(t)E$.
Then $F(p)$ is a $\C[E]$-Galois $\C[t]$-order.
\end{Example}

\begin{Example}
(Principal flag orders \cite{Web2019})
Let $\La$ be a noetherian integrally closed domain over $\K$ with field of fractions $L$,
$W$ a finite group acting faithfully on $\La$,
$\mathscr{M}$ a submonoid of $\Aut(\La)$ that is normalized by $W$,
and $\mathcal{F}$ be the skew monoid ring $L\rtimes (\mathscr{M}\rtimes W)$, which acts naturally on $L$.
A subring $F$ of $\mathcal{F}$ containing $\La\# W$ is a \emph{principal flag order} if 
$LF=\mathcal{F}$ and $X(\La)\subset\La$ for all $X\in F$. We may regard the monoid algebra $\K\mathscr{M}$ as a coideal subalgebra of the group algebra $\K\Aut_\K(L)$. Thus any principal flag order $F$ is a $\CH$-Galois $\La$-order with $\CH=(\K \mathscr{M})\rtimes W$ being a coideal subalgebra of the Hopf algebra $\tilde\CH=\K\Aut_\K(L)\rtimes W$.
\end{Example}

The following example was suggested to the author by Sam Evens.

\begin{Example}[Hecke algebras of symmetrizable Cartan data]
\label{ex:GKV}
We recall the construction from \cite{GinKapVas1997}.
Let $(A,\Pi,\Pi^\vee,P,P^\vee)$ be a symmetrizable Cartan datum: $A=(a_{ij})_{i,j=1}^r$ is a symmetrizable generalized Cartan matrix, $P$ is a free abelian group of rank $n=2r-\op{rank}A$, $P^\vee=\Hom_\Z(P,\Z)$, $\Pi=\{\al_i\}_{i=1}^r\subset P$ and $\Pi^\vee=\{\al_i^\vee\}_{i=1}^r\subset P^\vee$ are $\Z$-linearly independent subsets such that $\langle \al_i^\vee,\al_j\rangle=a_{ij}$.
Let $\mathscr{C}$ be the additive group of complex numbers $\C$, or the multiplicative group of nonzero complex numbers $\C^\ast$.
Put $T=\mathscr{C}\otimes_\Z P^\vee$. It is an abelian affine algebraic group.
Thus we have either $T=\C\otimes_\Z P^\vee \cong \C^n$ (Cartan subalgebra of Kac-Moody algebra) or
$T=\C^\ast\otimes_\Z P^\vee\cong (\C^\ast)^n$ (torus of Kac-Moody group).
Let $\CH=\C W$, the group algebra of $W$, and let $\La$ be the algebra of regular functions on $T$, $\La=\C[T]$. Then $L=\C(T)$, the field of rational functions on $T$.
The (possibly infinite) Weyl group $W$ acts on $P^\vee$, hence on $T$, hence on $\C[T]$.
Let $q\in \mathscr{C}$.
Let $\mathbf{H}_q$ denote the subalgebra of $\C(T)\#\C W$ generated by $\C[T]$ and the 
\emph{Lusztig-Demazure} operators $\si_i\in L\#\CH$ given by
\begin{equation}
\si_i = q\frac{t^{\al_i}s_i-1}{t^{\al_i}-1}-q^{-1}\frac{s_i-1}{t^{\al_i}-1},
\quad i=1,2,\ldots,r.
\end{equation}

\begin{Theorem}[Ginzburg-Kapranov-Vasserot {\cite{GinKapVas1997}}]
The algebras $\mathbf{H}_q$ have the following properties:
\begin{enumerate}[{\rm (i)}]
\item If $\mathscr{C}=\C^\ast$ (respectively $\mathscr{C}=\C$) and $A$ is of finite type, then $\mathbf{H}_q$ is isomorphic to the affine (respectively degenerate affine) Iwahori-Hecke algebra associated to the transpose of $A$;
\item If $\mathscr{C}=\C^\ast$ (respectively $\mathscr{C}=\C$) and $A$ is of affine type, then $\mathbf{H}_q$ is isomorphic to the double affine (respectively degenerate double affine) Hecke algebra;
\item $\mathbf{H}_q$ is free as a left $\C[T]$-module.
\end{enumerate}
\end{Theorem}

It is well-known (and easy to check) that the $\si_i$ preserve $\C[T]$. Therefore, $\mathbf{H}_q$ are examples of $\C W$-Galois $\La$-orders.
\end{Example}

\begin{Example}[Rational Cherednik algebras]
The rational Cherednik algebra associated to the complex reflection group $G(\ell,p,n)$ is a principal flag order \cite{LePWeb2019}. We note here that the rational Cherednik algebra (at $t\neq 0$) associated to any finite group may be regarded as Hopf Galois order in the rational-differential setting.

Let $V$ be a finite-dimensional complex vector space, $G$ a finite subgroup of $\GL(V)$, and $S=\{g\in G\mid \codim\ker(g-1)=1\}$ the set of complex reflections in $G$, $t\in\C\setminus\{0\}$,  and $c:S\to\C$ a $G$-invariant function. Let $\mathcal{D}(V)_\mathrm{r}$ be the algebra of differential operators on $V$ with rational function coefficients. Equivalently, $\mathcal{D}(V)_\mathrm{r}$ can be identified with the smash product $\C(V)\# S(V)$, where $\C(V)$ is the field of fractions of the algebra $\C[V]$ of polynomial functions on $V$ and $S(V)$ is the commutative Hopf algebra of constant coefficient differential operators on $V$. 
As mentioned in \cite[Section 2.6]{Eti2017}, using the Dunkl-Opdam representation \cite{DunOpd2003}, the rational Cherednik algebra $H_{t,c}(V,G)$ is isomorphic to the subalgebra of $\C(V)\# S(V)\rtimes G$ generated by $\C[V]\# \C G$ and $\{D_y\mid y\in V\}$ where $D_y$ are the Dunkl-Opdam operators
\begin{equation}
D_y=t\frac{\partial}{\partial y}+\sum_{s\in S}\frac{2c_s}{1-\la_s} \frac{(\al_s,y)}{\al_s}(s-1),
\end{equation}
where $\al_s\in V^\ast$ is a nonzero linear functional vanishing on the fixed hyperplane of $s$, $\al_s^\vee\in V$ the element vanishing on the fixed hyperplane of $s$ in $V^\ast$ such that $(\al_s,\al_s^\vee)=2$,
$\la_s$ is the nontrivial eigenvalue for $s$ on $V^\ast$.

We note here that $H_{c,t}(V,G)$ is an $\big(S(V)\rtimes G\big)$-Galois $\C[V]$-order. First, we saw in Example \ref{ex:rational-differential} that $\big(S(V)\rtimes G, \C[V]\big)$ is a setting.
We have $\C[V]\subset H_{c,t}(V,G)$ by definition.
As is well-known, since $\al_s$ vanishes on $\ker(s-1)$, 
\begin{equation}
D_y(f)\in\C[V],\qquad \forall f\in \C[V], y\in V.
\end{equation}
Therefore $X(\C[V])\subset\C[V]$ for all $X\in H_{c,t}(V,G)$.
Lastly, since $t\neq 0$ and $\C[V]\#\C G\subset H_{c,t}(V,G)$, we have $\partial/\partial y \in \C(V) H_{c,t}(V,G)$, hence $\C(V)H_{c,t}(V,G)=\C(V)\#S(V)\rtimes G$.
\end{Example}

We continue by stating and proving some useful Lemmas. For the rest of this section we assume that $(\CH,\La)$ is a setting and that $F$ is an $\CH$-Galois $\La$-order, and $L$ denotes the fraction field of $\La$.

\begin{Lemma}
\label{lem:direct-sum}
Let $F_-=\{X\in F\mid \hat X(1_\La)=0\}$. Then $F_-$ is a left $\La$-submodule of $F$ and
\begin{equation}
F=\La\oplus F_-.
\end{equation}
\end{Lemma}

\begin{proof}
If $X\in F$, put $X_-=X-\hat X(1_\La)$. Then $\hat X_-(1_\La)=\hat X(1_\La)-\hat X(1_\La)1_\La=0$ so $X_-\in F_-$. Clearly $X=\hat X(1_\La) + (X-\hat X(1_\La))$ which shows that $F=\La+F_-$. Suppose $Y\in \La\cap F_-$. Then $Y=\hat Y(1_\La)$ since $Y\in\La$. On the other hand,  $\hat Y(1_\La)=0$ since $Y\in F_-$. Therefore $Y=0$.
\end{proof}

\begin{Lemma} \label{lem:maxcomm}
$\La$ is a maximal commutative subalgebra of $F$.
\end{Lemma}

\begin{proof}
Let $X\in F$ be such that $Xa=aX$ for all $a\in\La$. By Lemma \ref{lem:direct-sum} it suffices to show that $X_-=0$, where $X_-=X-\hat X(1_\La)$. For all $a\in \La$ we have
\[\hat{X_-}(a) = \hat X(a)-\hat X(1_\La)\cdot a = (\hat X\hat a)(1_\La)-(\hat a\hat X)(1_\La) =(\hat X\hat a-\hat a\hat X)(1_\La)=0.\]
Since the action of $F$ on $\La$ is faithful, we conclude that $X_-=0$.
\end{proof}

\begin{Lemma} \label{lem:center}
The center of $F$ equals
\begin{equation}\label{eq:center}
Z(F)=\{a\in\La\mid \forall x\in\CH:\, xa=ax\}=\{a\in\La\mid \forall x\in\CH:\,(x_{(1)}\tri a) x_{(2)}=ax\}.
\end{equation}
In particular, when $\CH$ is a Hopf algebra,
\begin{equation}\label{eq:center2}
Z(F)=\La^\CH\defby\{a\in\La\mid \forall x\in\CH:\, x\tri a=\ep(x)a\}.
\end{equation}
\end{Lemma}

\begin{proof}
By Lemma \ref{lem:maxcomm}, $Z(F)\subset\La$. Since $F$ is an $\CH$-Galois $\La$-order we have $L F=L\#\CH$ which implies the leftmost equality in \eqref{eq:center}.
The second equality in \eqref{eq:center} follows the smash relation \eqref{eq:smash}.
Using $x\tri a= x_{(1)}a S(x_{(2)})$ for $x\in\CH$ and $a\in\La$, \eqref{eq:center2} follows.
\end{proof}

\begin{Lemma}
\label{lem:bimodule}
The following statements hold.
\begin{enumerate}[{\rm (i)}]
\item $aS(x)=S(x_{(1)})\cdot (x_{(2)}\tri a)$ for any $a\in L$ and $x\in\tilde\CH$.
\item $ax = x_{(2)}\cdot \big(S^{-1}(x_{(1)})\tri a\big)$ for any $a\in L$ and $x\in\tilde\CH$.
\item $\La\CC=\CC\La$ for any left coideal $\CC$ of $\tilde\CH$.
\item $L\CC=\CC L$ for any left coideal $\CC$ of $\tilde\CH$.
\end{enumerate}
\end{Lemma}

\begin{proof}
(i) We have
\begin{align*}
aS(x) &=a S\big(\ep(x_{(1)})x_{(2)}\big) 
 &\text{(counit axiom)}\\
&=\ep(x_{(1)})\cdot a\cdot S(x_{(2)}) 
 &\text{(linearity)}\\
&=S(x_{(1)})\cdot x_{(2)}\cdot a\cdot S(x_{(3)})
 &\text{(antipode axiom)}\\
&=S(x_{(1)})\cdot(x_{(2)}\tri a)\cdot x_{(3)}\cdot S(x_{(4)})
 &\text{(smash relation \eqref{eq:smash})}\\
&=S(x_{(1)})\cdot (x_{(2)}\tri a)\cdot \ep(x_{(3)})
 &\text{(antipode axiom)}\\
&=S(x_{(1)})\cdot (x_{(2)}\tri a)
 &\text{(linearity and counit axiom)}
\end{align*}

\noindent (ii) Put $y=S^{-1}(x)$, apply part (i) with $y$ in place of $x$, and use $(S\otimes S) \circ \Delta=\Delta^{\mathrm{op}}\circ S$.

\noindent (iii) Since $\CC$ is a left coideal of $\tilde\CH$, by \eqref{eq:smash} we have $\CC\La\subset \La\CC$. By part (ii), $\La\CC\subset\CC\La$.

\noindent (iv) Since $L$ is the field of fractions of $\La$, this is immediate from part (iii).
\end{proof}

\begin{Definition}[{\cite{DroFutOvs1994}}]
A commutative subalgebra $B$ of an algebra $A$ is a \emph{Harish-Chandra subalgebra} of $A$ if every finitely generated $(B,B)$-subbimodule of $A$ is finitely generated as a left, and as a right, $B$-module. 
\end{Definition}
 
\begin{Lemma}
$\La$ is a Harish-Chandra subalgebra of $F$.
\end{Lemma}

\begin{proof}
It suffices to prove that $\La$ is a Harish-Chandra subalgebra of $L\#\CH$. Furthermore, 
since $\La xy \La $ is a homomorphic image of $\La x \La \otimes_\La \La y\La$ and $\La (x+y)\La\subset \La x \La +\La y\La$, it suffices to check $\La x\La$ is finitely generated, on the left and right, for $x$ belonging to a generating set of $L\#\CH$ as a ring. We choose the generating set $L\cup\CH$. For $x\in L$ the statement is trivial since $L$ is commutative.
Let $x\in\CH$. By the smash relation \eqref{eq:smash}, $\La x\La$ is contained in the left $\La$-submodule of $\La\#\CH$ generated by the finitely many elements $x_{(2)}\in\CH$. 
Since $\La$ is noetherian, $\La x\La$ is itself finitely generated as a left $\La$-module.
Similarly, by Lemma \ref{lem:bimodule}(ii), $\La x\La$ is finitely generated as a right $\La$-module.
\end{proof}

\begin{Lemma} \label{lem:FO-det}
Let $\{X_1,\ldots,X_n\}$ be a finite linearly independent (on the left, over $L$) subset of $L\#\CH$. Then there exist $a_1,a_2,\ldots,a_n\in \La$ such that the determinant of the matrix $\big(X_i(a_j)\big)_{i,j=1}^n$ is nonzero.
\end{Lemma}

\begin{proof}
Tensoring $\La\#\CH\hookrightarrow \End(\La), X\mapsto \hat X$, from the left by the flat $\La$-module $L$ gives an injective map
$L\#\CH\hookrightarrow \Hom(\La,L)$.
The rest is exactly as in part (ii) of \cite{Har2020}. 
\end{proof}

The following result is the main result of this section. It is a crucial step towards establishing Theorem \ref{thm:main} about finiteness of fibers. It is a generalization of \cite[Thm.~2.21]{Har2020}. We call it the \emph{Futorny-Ovsienko property} because the statement was part of the original definition of  Galois orders from \cite{FutOvs2010}.

\begin{Theorem}[Futorny-Ovsienko property]
 \label{thm:FO}
Let $(\CH,\La)$ be a setting and $F$ be an $\CH$-Galois $\La$-order.
Let $V$ be any finite-dimensional left (right) $L$-subspace of $L\#\CH$. Then $V\cap F$ is a finitely-generated left (right) $\La$-module.
\end{Theorem}

\begin{proof}
Let $\{X_1,X_2,\ldots,X_n\}$ be a left $L$-basis for $V$. After rescaling if necessary, we may assume without loss of generality that $X_i\in\La\#\CH$ for all $i=1,2,\ldots,n$. By Lemma \ref{lem:FO-det}, there exists $\{a_1,a_2,\ldots,a_n\}\subset\La$ such that $d\defby \det\big(\hat{X_i}(a_j)\big)_{i,j=1}^n \in\La\setminus\{0\}$. We claim that $F\cap V\subset\frac{1}{d}(\La X_1+\La X_2+\cdots+\La X_n)$. Let $X\in F\cap V$. Since $X\in V$ there are $b_1,b_2,\ldots,b_n\in L$ such that
\[
X=b_1X_1+b_2X_2+\cdots b_nX_n.
\]
Since $X\in F$, $\hat{X}(a)\in\La$ for all $a\in\La$. In particular
\[
c_i\defby \hat X(a_i)=b_1 \hat{X_1}(a_i)+b_2\hat{X_2}(a_i)+\cdots+b_n\hat{X_n}(a_i)\in\La.
\]
Inverting the matrix $A=\big(X_i(a_i)\big)_{i,j=1}^n$ and using that $A^{-1}$ has entries from $\frac{1}{d}\La$, we conclude that $b_i\in\frac{1}{d}\La$ for all $i$.
Since $\frac{1}{d}(\La X_1+\La X_2+\cdots \La X_n)$ is generated by the finite set $\{\frac{1}{d} X_1, \frac{1}{d}X_2,\ldots,\frac{1}{d}X_n\}$ as a left $\La$-module, and $\La$ is noetherian, $F\cap V$ is also finitely generated as a left $\La$-module.

For the right-handed case, let $V$ be a finite-dimensional right $L$-subspace of $L\#\CH$. 
Since $L\#\CH=\CH L$ by Lemma \ref{lem:bimodule}(iv), there 
exists a finite-dimensional $\K$-subspace $\CH'$ of $\CH$ such that $V\subset \CH' L$. Furthermore, by the Finiteness Theorem for Coalgebras (see e.g. \cite{Mon1993}), there exists a finite-dimensional left coideal $C$ of $\tilde\CH$ such that $\CH'\subset C\subset \CH$.
Since $\La$ is noetherian it suffices to show that $F\cap (CL)$ is contained in a finitely generated right $\La$-module.
Let $\{c^1, c^2,\ldots, c^n\}$ be a $\K$-basis for $C$. 
Let $X\in F\cap (CL)$ and write
\[X = c^1 b_1+c^2 b_2+ \cdots + c^n b_n\]
for some $b_i\in L$.
By the smash relation \eqref{eq:smash} in $L\#\CH$,
\begin{align*}
X&= ( c^1_{(1)}\tri b_1) c^1_{(2)} + 
( c^2_{(1)}\tri b_1) c^2_{(2)} + \cdots
( c^m_{(1)}\tri b_1) c^n_{(2)}\\
&= \tilde b_1 c^1+ \tilde b_2 c^2+\cdots +\tilde b_n c^n
\end{align*}
for some $\tilde b_i\in L$, since $C$ is a left coideal of $\tilde\CH$. 
By Lemma \ref{lem:FO-det}, there exists $\{a_1,a_2,\ldots,a_n\}\subset\La$ such that $d\defby \det\big(c^i\tri a_j\big)_{i,j=1}^n \in\La\setminus\{0\}$.
Since $X\in F$ and $F$ is an $\CH$-Galois $\La$-order, we have $\hat X(a_i)\in\La$ for each $i$.
As in the left-handed case, this implies that 
 $\tilde b_i\in\frac{1}{d}\La$ for all $i$. Therefore, $X\in \frac{1}{d}\La C$. By Lemma \ref{lem:bimodule}(iii), $\La C = C\La$.
This shows that $F\cap (CL)\subset \frac{1}{d}C\La$, which is a finitely generated right $\La$-module.\end{proof}

\section{Harish-Chandra modules}\label{sec:HC-modules}

\subsection{Generalities}

\begin{Definition}[{\cite{DroFutOvs1994}}]
An $F$-module $V$ is a \emph{Harish-Chandra module} if
\begin{equation}
V=\bigoplus V^\Fm,\qquad  V^\Fm=\{v\in V\mid \Fm^N v=0, N\gg 0\},
\end{equation}
where the sum is over the set $\cfs(\La)$ of all maximal ideals $\Fm$ in $\La$ of finite codimension. Let $\mathrm{Irr}(F,\Fm)$ be the set of isomorphism classes of simple Harish-Chandra $F$-modules $V$ for which $V^\Fm\neq 0$.
\end{Definition}

\begin{Notation} 
Fix $\Fm\in\MaxSpec(\La)$. Put
\begin{equation}
\hat{F}_\Fm = \lim_{\longleftarrow} F/(F\Fm^N+\Fm^NF) 
= \lim \big(F/(F\Fm+\Fm F) \leftarrow F/(F\Fm^2+\Fm^2F) \leftarrow\cdots \big).
\end{equation}
Equip $\hat{F}_\Fm$ with the inverse limit topology, defined as the topology with fewest open sets such all cosets of the kernel of $\pi_N:\hat F_\Fm\to F/(F\Fm^N+\Fm^NF)$ are open for all $N$.
A left $\hat F_\Fm$-module $M$ is \emph{discrete} if the action $\hat F_\Fm\times M\to M$ is continuous when $M$ is given the discrete topology. This is equivalent to that for every $v\in M$ there exists an $N$ such that the map $\hat F_\Fm\to M$ given by $Y\mapsto Yv$ factors through $F/(F\Fm^N+\Fm^NF)$. 
Let $\mathrm{Irr}^\mathrm{d}(\hat F_\Fm)$ denote the set of isomorphism classes of simple discrete left $\hat F_\Fm$-modules. Similarly for right modules.
\end{Notation}

\begin{Theorem}[{\cite{DroFutOvs1994}}]
For any $\Fm\in\cfs(\La)$, the map $[V]\mapsto [V_\Fm]$ is a bijection between $\mathrm{Irr}(F,\Fm)$ and $\mathrm{Irr}^\mathrm{d}(\hat F_\Fm)$.
\end{Theorem}

\begin{Theorem}[{\cite{DroFutOvs1994}}]
Let $\Fm\in\cfs(\La)$ and put $\displaystyle \La_\Fm=\lim_{\longleftarrow}\La/\Fm^N$.
The main theorem of \cite{DroFutOvs1994} states that if $\hat F_\Fm$ is finitely generated as a  right $\La_\Fm$-module, then
\begin{enumerate}[{\rm (i)}]
\item $\op{Irr}(F,\Fm)$ is finite,
\item For any $[V]\in\op{Irr}(F,\Fm)$, $V^\Fm$ is finite-dimensional.
\end{enumerate}
\end{Theorem}

\subsection{Existence of Harish-Chandra modules over $\CH$-Galois $\La$-orders} We first prove an existence result for general $\Fm$.

\begin{Theorem}\label{thm:nonempty}
Let $(\CH,\La)$ be a setting, $F$ be an $\CH$-Galois $\La$-order,
and $\Fm$ be a maximal ideal of $\La$. 
\begin{enumerate}[{\rm (i)}]
\item There exists a maximal right ideal $J$ of $F$ containing $\Fm$.
\item There exists a simple left $F^{\op{op}}$-module $M$ with $M^\Fm\neq 0$.
\item If $\Fm$ has finite codimension, then $\op{Irr}(F^{\op{op}},\Fm)\neq\emptyset$.
\end{enumerate}
\end{Theorem}

\begin{proof}
(i) It suffices to show that $\Fm F$ is a proper right ideal of $F$. Let $\psi_\Fm: F\to\La/\Fm$ be the composition of the left $\La$-module epimorphism $F\to\La$, $X\mapsto\hat X(1_\La)$ and the canonical projection $\La\to\La/\Fm$. Then $\Fm F\subset\ker(\psi_\Fm)$ and we get an induced map $\bar\psi_\Fm:F/\Fm F\to \La/\Fm$ which is an epimorphism  of left $\La$-modules. In particular $F/\Fm F$ is nonzero.

(ii) Take $M=F/J$ where $J$ is as in (i). Then $1+J \in M^\Fm\setminus\{0\}$.

(iii) Any module generated by a generalized weight vector for $\La$ of weight $\Fm\in\cfs(\La)$ is a Harish-Chandra module, by \cite[Prop.~14]{DroFutOvs1994}. Thus (iii) follows from (ii).
\end{proof}

\subsection{Canonical modules of local distributions}
We refer to \cite[Chapter I.7]{Jan1987} for an algebraic treatment of local distributions.

Let $(\CH,\La)$ be a setting and $F$ be an $\CH$-Galois $\La$-order. Then $\La$ is a left $F$-module, hence the dual space 
$\La^\ast=\Hom_\K(\La,\K)$ is a left $F^{\op{op}}$-module with action
\begin{equation}\label{ref:right-canonical}
(X\xi)(a)=\xi \big( \hat X(a)\big),\quad \forall\xi\in\La^\ast,\, X\in F^{\op{op}},\, a\in \La.
\end{equation}
For $\Fm\in \cfs(\La)$, we let $\op{Dist}(\La,\Fm)$ be the subspace of $\La^\ast$ of all linear functionals $\xi$ such that $\Fm^n\subset\ker\xi$ for $n\gg 0$.
Equivalently, $\op{Dist}(\La,\Fm)=\Hom_\K^{\mathrm{d}}(\hat\La_\Fm,\K)$ the set of linear functionals on $\hat\La_\Fm$ that are continuous with respect to pro-finite topology on $\hat\La_\Fm$ and discrete topology on $\K$.
 Put $\op{Dist}(\La)=\bigoplus_\Fm \op{Dist}(\La,\Fm)$ where $\Fm$ runs over $\cfs(\La)$.

\begin{Theorem} \label{thm:canonical}
Let $F$ be an $\CH$-Galois $\La$-order.
\begin{enumerate}[{\rm (i)}]
\item $\op{Dist}(\La)$ is an $F^{\op{op}}$-submodule of $\La^\ast$.
\item $\op{Dist}(\La)$ is a Harish-Chandra $F^{\op{op}}$-module with respect to $\La$.
\item If $\la:\La\to\K$ is an algebra map with kernel $\Fm_\la$, then the cyclic left $F^{\op{op}}$-submodule of $\op{Dist}(\La)$ generated by $\la$ has a unique simple quotient $V(\la)$. Moreover $V(\la)$ is a simple Harish-Chandra $F^{\op{op}}$-module with $V^{\Fm_\la}\neq 0$.
\end{enumerate}
\end{Theorem}

\begin{proof}
$\op{Dist}(\La,\Fm)$ is nothing but the generalized weight space in $\La^\ast$ of weight $\Fm$. Since $\La$ is a Harish-Chandra subalgebra of $F$, (i) and (ii) follow directly from \cite[Prop.~14]{DroFutOvs1994}.

(iii) Following the argument from \cite[Prop.~2]{Nil2016}, as used in \cite[Prop.~5]{EarMazVis2020}, we have
\[
\Hom_\La(\La/\Fm_\la, \La^\ast)\cong\Hom_\K(\La\otimes_\La (\La/\Fm_\la), \K)\cong\K.
\]
Thus the $\Fm_\la$-weight space $\op{Dist}(\La)_{\Fm_\la} = (\La^\ast)_{\Fm_\la} = \{\xi\in\La^\ast\mid \xi\Fm_\la=0\}$ is one-dimensional and is spanned by $\la$. Therefore any $F$-submodule $N$ of $\op{Dist}(\La)$ with $N^{\Fm_\la}\neq 0$ (hence $N_{\Fm_\la}\neq 0$) contains the cyclic submodule $\la F$. This shows that the sum of all proper right $F$-submodules of $\la F$ is itself proper, and equals the unique maximal $F$-submodule.
\end{proof}

\subsection{Superfluous subcoalgebras and the stabilizer coalgebra} \label{sec:stabilizer}

When $\CH$ is a group algebra $\K G$ acting on $\La$, it acts on the maximal spectrum of $\La$, and the stabilizer subgroup $\op{Stab}(G,\Fm)=\{g\in G\mid g(\Fm)=\Fm\}$ spans a subalgebra which is a Hopf subalgebra of $\K G$.
The dimension of this subalgebra (i.e. the order of the subgroup $\op{Stab}(G,\Fm)$), plays a crucial role in the finiteness theorem of \cite{Web2019}.
For more general Hopf algebras (and coideal subalgebras) $\CH$, giving meaning to ``the stabilizer'' $\op{Stab}(\CH,\Fm)$ is a bit trickier. 
In \cite{Sch1990}, such stabilizers are defined in a setting of Hopf-Galois extensions and the result is a subcoalgebra of $\CH$.
However in the context of the present paper (which is incomparable to that of \cite{Sch1990}), it  more natural for us to define the stabilizer $\op{Stab}(\CH,\Fm)$ as a \emph{quotient} of $\CH$. 
Therefore we first consider a certain ``superfluous coradical'' $\CS(\CH,\Fm)$, by which we mod out in order to obtain $\op{Stab}(\CH,\Fm)$.
When $\CH$ is a group algebra $\K G$, then $\CS(\CH,\Fm)$ is the span of the complement of $\op{Stab}(G,\Fm)$ in $G$.

\begin{Definition} Let $(\CH,\La)$ be a setting,
 $\CC$ be a subcoalgebra of $\tilde\CH$ and $\Fm\in\MaxSpec(\La)$.
\begin{enumerate}[{\rm (i)}]
\item An element $R=\sum_i r_i\otimes s_i\in\La\otimes\La$ is a \emph{reductor for $\CC$ mod $\Fm$} if
\begin{enumerate}[{\rm (R1)}]
\item The image of $R$ in $\La\otimes\La/(\La\otimes\Fm+\Fm\otimes\La)$ is invertible,
\item $\sum_i r_i\cdot (x\tri s_i) = 0$ for all $x\in\CC$.
\end{enumerate}
\item $\CC$ is called $\emph{superfluous mod $\Fm$}$ if for every finite-dimensional subcoalgebra $\CC'$ of $\CC$ there exists a reductor for $\CC'$ mod $\Fm$.
\end{enumerate}
\end{Definition}

\begin{Remark}
If $\CC$ is superfluous mod $\Fm$ then the map $\CC\to (\La/\Fm)\otimes_\La \La \CC\La \otimes_\La (\La/\Fm)$ (sending $x$ to $1\otimes x \otimes 1$) is the zero map, which motivates the terminology. Indeed, for any $x\in\CC$ we have $1\otimes x\otimes 1=1\otimes \sum_{i,j} r_i r_j' x s_i s_j'\otimes 1=1\otimes\sum_j r_j'(\sum_i r_i x_{(1)}\tri s_i)x_{(2)}s_j'\otimes 1=0$ if $R=\sum_i r_i\otimes s_i$ is a reductor for a finite-dimensional subcoalgebra of $\CC$ containing $x$, and $R'=\sum_j r_j'\otimes s_i'\in\La\otimes\La$ is an inverse for $R$ mod $\Fm\otimes\La+\La\otimes\Fm$.
\end{Remark}

\begin{Example}
If $\CC=\K g$ where $g$ is a grouplike element of $\tilde\CH$, then $\CC$ is superfluous mod $\Fm$ iff $g\tri\Fm\neq\Fm$.
\end{Example}

\begin{Lemma}\label{lem:superfluous}
If $\CC_1$ and $\CC_2$ are two subcoalgebras of $\tilde\CH$ that are superfluous mod $\Fm$ then $\CC_1+\CC_2$ is also superfluous mod $\Fm$.
\end{Lemma}

\begin{proof}
To simplify the proof we first reformulate condition (R2).
Let $.$ denote the action of $\tilde\CH$ on $\La\otimes\La$ given by
\begin{equation}
x.(a\otimes b) = a \otimes (x\tri b),\quad\forall x\in\tilde\CH, \; a,b\in\La.
\end{equation}
This makes $\La\otimes\La$ an $\tilde\CH$-module algebra. In particular,
\begin{equation}
x.(RS)=(x_{(1)}.R)(x_{(2)}.S),\quad\forall x\in\tilde\CH, \; R,S\in\La\otimes\La.
\end{equation}
Let $\mu_\La:\La\otimes\La\to\La$ denote the multiplication map.
Then condition (R2) is equivalent to
\begin{equation}
\CC.R\subset\ker \mu_\La.
\end{equation}
For $i=1,2$, let $R_i$ be a reductor for $\CC_i$ mod $\Fm$. Consider the product $R=R_1R_2$. Clearly $R$ satisfies condition (R1) in the definition above. We have
\[(\CC_1+\CC_2).R \subset \CC_1.(R_1R_2)+\CC_2.(R_1R_2)
\subset (\CC_1.R_1)(\CC_1.R_2)+(\CC_2.R_1)(\CC_2.R_2)\subset \ker \mu_\Lambda\]
since $\CC_i.R_i\subset\ker \mu_\La$ for $i=1,2$ and $\ker \mu_\La$ is an ideal of $\La\otimes\La$.
\end{proof}

\begin{Definition}
Let $\CS(\tilde\CH,\Fm)$ denote the sum of all subcoalgebras of $\tilde\CH$ that are superfluous mod $\Fm$.
\end{Definition}

\begin{Corollary}
$\CS(\tilde\CH,\Fm)$ is superfluous mod $\Fm$. Consequently $\CS(\tilde\CH,\Fm)$ is the unique maximal (with respect to inclusion) element in the family of all subcoalgebras of $\tilde\CH$ that are superfluous mod $\Fm$.
\end{Corollary}

\begin{proof}
Let $\CC$ be any finite-dimensional subcoalgebra of $\CS(\tilde\CH,\Fm)$. Then $\CC$ is contained in a finite sum $\CD$ of   subcoalgebras of $\tilde\CH$ that are superfluous mod $\Fm$. By induction and Lemma \ref{lem:superfluous}, $\CD$ is superfluous mod $\Fm$, hence so is $\CC$.
\end{proof}

\begin{Definition}
If $\CC\subset\tilde\CH$ is any left coideal, we define
\[\CS(\CC,\Fm)=\big\{x\in\CC\mid\Delta(x)\in\CS(\tilde\CH,\Fm)\otimes\CC\big\}.\]
By coassociativity, $\CS(\CC,\Fm)$ is a left coideal of $\tilde\CH$ contained in $\CC$.
\end{Definition}

\begin{Definition} Given a setting $(\CH,\La)$ and a maximal ideal $\Fm$ of $\La$, we define the \emph{stabilizer (of $\CH$ at $\Fm$)} to be the left $\tilde\CH$-comodule
\[\op{Stab}(\CH,\Fm) = \CH / \CS(\CH,\Fm).\]
In the case when $\CH=\tilde\CH$, $\op{Stab}(\CH,\Fm)$ is a coalgebra.
\end{Definition}

\begin{Example} When $\CH$ is a group algebra $\K G$ we have
$\CS(\K G, \Fm) = \bigoplus_{g\in G, g(\Fm)\neq\Fm}\K g$, and hence
\begin{equation}
\op{Stab}(\K G, \Fm) \cong \K\Stab(G,\Fm),
\end{equation}
where $\Stab(G,\Fm)$ is the stabilizer subgroup of $G$ at $\Fm$.
\end{Example}

The following useful result is analogous to \cite[Thm.~4.6(a)]{Sch1990}.

\begin{Lemma}
A subcoalgebra $\CC$ of $\tilde\CH$ is superfluous mod $\Fm$ iff $\tilde\CH_0\cap\CC$ is superfluous mod $\Fm$, where $\tilde\CH_0$ is the coradical of $\tilde\CH$.
\end{Lemma}

\begin{proof}
If $\CC$ is superfluous mod $\Fm$, then every subcoalgebra of $\CC$ is too, in particular $\tilde\CH_0\cap\CC$.

Conversely, suppose $\CC$ is a subcoalgebra of $\tilde\CH$ such that $\tilde\CH_0\cap\CC$ is superfluous mod $\Fm$. We will prove that $\CC$ is superfluous mod $\Fm$. Without loss of generality we may assume $\CC$ is finite-dimensional. Let $\CC_k=\tilde\CH_k\cap \CC$ where $\{\tilde\CH_k\}_{k=0}^\infty$ is the coradical filtration of $\tilde\CH$. It suffices to prove that each subcoalgebra $\CC_k$ is superfluous mod $\Fm$. We do this by induction on $k$. For $k=0$ this is true by assumption on $\CC$. Let $k>0$. By assumption on $\CC$ there is a reductor $R_0$ for $\CC_0$ mod $\Fm$. By the induction hypothesis there is a reductor $R_{k-1}$ for $\CC_{k-1}$ mod $\Fm$. Using the notation as in the proof of Lemma \ref{lem:superfluous}, $\CC_0.R_0\subset\ker\mu_\La$ and $\CC_i.R_{k-1}\subset\ker\mu_\La$ for $i=1,2,\ldots,k-1$ and hence by $\Delta(\CC_k)\subset \CC_0\otimes\CC_k+\CC_k\otimes\CC_{k-1}$ we have
\[\CC_k.(R_0R_{k-1}) \subset (\CC_0.R_0)(\CC_k.R_{k-1}) + (\CC_k.R_0)(\CC_{k-1}.R_{k-1})\subset \ker\mu_\La.\]
This shows that $R_0R_{k-1}$ is a reductor for $\CC_k$ mod $\Fm$. 
\end{proof}

\begin{Corollary}
$\CS(\tilde\CH,\Fm)$ equals the sum of all subcoalgebras $\CC$ of $\tilde\CH$ such that $\tilde\CH_0\cap\CC$ is superfluous mod $\Fm$, where $\tilde\CH_0$ is the coradical of $\tilde\CH$.
\end{Corollary}

Recall that a Hopf algebra is \emph{pointed} if every simple subcoalgebra is one-dimensional, or equivalently the coradical coincides with the group algebra of the set of grouplike elements.

\begin{Corollary}
If $\tilde\CH$ is a pointed Hopf algebra with coradical $\K G$ then $\CS(\tilde\CH,\Fm)$ is the sum of subcoalgebras $\CC$ such that $\CC\cap \K G \subset \bigoplus_{g\in G,\, g(\Fm)\neq\Fm} \K g$.
\end{Corollary}

Recall that a Hopf algebra is \emph{connected} if its coradical is one-dimensional.
\begin{Corollary}
Suppose that $\tilde\CH=\CH'\rtimes W$ where $\CH'$ is a connected Hopf subalgebra and $W$ is a group. Then for all $\Fm\in\MaxSpec(\La)$,
$
\CS(\tilde\CH,\Fm) = \CH'\otimes\bigoplus_{w\in W, w(\Fm)\neq \Fm} \K w
$
and hence
\begin{equation}
\op{Stab}(\tilde\CH,\Fm) \cong \CH'\otimes \K\Stab(W,\Fm).
\end{equation}
\end{Corollary}

\begin{Example}
For $(\CH,\La)=(U(\Fg)\rtimes W, \C[X])$ as in Example \ref{ex:example2}, we have for any $\la\in X$
\begin{equation}
\op{Stab}(\CH,\Fm_\la) \cong U(\Fg)\otimes \K\Stab(W,\Fm_\la).
\end{equation}
\end{Example}

\subsection{Finiteness condition}

\begin{Lemma}\label{lem:finite-fibers}
Let $(\CH,\La)$ be a setting and $\Fm\in\MaxSpec(\La)$ 
such that the stabilizer $\op{Stab}(\CH,\Fm)$ is finite-dimensional. Then for any $\CH$-Galois $\La$-order $F$, the completion $\hat{F}_\Fm$ is finitely generated as a left and right $\hat\La_\Fm$-module. 
\end{Lemma}

\begin{proof}
Let $\CH'$ be a finite-dimensional linear complement to $\CS(\CH,\Fm)$ in $\CH$. By the finiteness theorem for coalgebras, we may without loss of generality assume that $\CH'$ is a finite-dimensional left $\tilde\CH$-coideal contained in $\CH$.
Put $F'=F\cap L\CH'$. Since $\CH'$ is a left coideal, $F'$ is a $(\La,\La)$-subbimodule of $F$.
By the Futorny-Ovsienko property (Theorem \ref{thm:FO}), $F'$ is finitely generated as a left and right $\La$-module. Thus it suffices to show that $F'$ surjects onto $F/(F\Fm^N+\Fm^NF)$ for all $N>0$.

Let $X\in F$ be arbitrary. Write
$X=f_1x^1+f_2x^2+\cdots+f_nx^n+g_1y^1+g_2y^2+\cdots+g_my^m$
where $x^i\in\CH'$ and $y^j\in \CS(\CH,\Fm)$ (superscripts being indices, not powers) and $f_i,g_j\in L$.
Let $\CC$ be a finite-dimensional subcoalgebra of $\CS(\tilde\CH,\Fm)$ such that $\Delta(y_j)\in\CC\otimes\CH$ for all $j$.
Pick a reductor for $\CC$ mod $\Fm$, say
$R=\sum_k a_k\otimes b_k\in \La\otimes\La$.
Let us regard $F$ as a left $\La\otimes\La$-module via $(a\otimes b)\cdot X = aXb$ for $a,b\in\La$ and $X\in F$.
Consider
\begin{align*}
Y &= R\cdot X =\\
&=\sum_k a_k X b_k \\
&= \sum_i f_i\sum_k a_k(x^i_{(1)}\tri b_k) x^i_{(2)}
+\sum_j g_j\sum_k a_k (y^j_{(1)}\tri b_k)y^j_{(2)}
\end{align*}
The sum over $j$ is zero because all $y^j_{(1)}\in\CC$ and $R$ is a reductor for $\CC$ mod $\Fm$.
Since $\CH'$ is a left $\tilde\CH$-coideal, 
all $x^i_{(2)}\in\CH'$. Thus $Y\in F'$.

On the other hand, since $R$ is invertible mod $\La\otimes\Fm+\Fm\otimes\La$, it is invertible mod $\La\otimes\Fm^N+\Fm^N\otimes\La$ for any $N>0$.
So for every $N>0$ there is an element $\bar R_N\in\La\otimes\La$ such that $\bar R_N R-1\otimes 1\in \La\otimes\Fm^N+\Fm^N\otimes\La$.
Put $X' = \bar R_N R\cdot X$.
Since $R\cdot X\in F'$ and $F'$ is a $(\La,\La)$-subbimodule of $F$, we have $X'\in F'$. On the other hand
 $X-X'=(1\otimes 1-\bar R_N R)\cdot X\in (\La\otimes\Fm^N+\Fm^N\otimes\La)\cdot X\subset F\Fm^N+\Fm^N F$.
In other words, the coset $X+F\Fm^N+\Fm^N F$ is in the image of $F'$ under the canonical projection  $F\to F/(F\Fm^N+\Fm^N F)$. Since $X$ was arbitrary this proves the claim.
\end{proof}

The above lemma can be generalized slightly. We will need this in Section \ref{sec:loc-fin}.

\begin{Lemma}\label{lem:finite-fibers-general}
Let $(\CH,\La)$ be a setting,
$\CC$ a left coideal of $\tilde\CH$ contained in $\CH$,
and $\Fm\in\MaxSpec(\La)$ 
such that the stabilizer $\op{Stab}(\CC,\Fm)$ is finite-dimensional. Then for any $\CH$-Galois $\La$-order $F$, the completion $\hat{F(\CC)}_\Fm$ where $F(\CC)=F\cap L\CC$, is finitely generated as a left and right $\hat\La_\Fm$-module. 
\end{Lemma}

\begin{proof}
Replace $\CH$ by $\CC$ in the proof of Lemma \ref{lem:finite-fibers}. The only difference is that $F(\CC)$ is not an algebra but just a $(\La,\La)$-subbimodule of $F$, but all the arguments go through without change.
\end{proof}

\subsection{Main finiteness theorem}

\begin{Theorem}\label{thm:main}
Let $(\CH,\La)$ be a setting and $\Fm\in\cfs(\La)$. Assume that that the stabilizer $\op{Stab}(\CH,\Fm)$ is finite-dimensional. Then for every $\CH$-Galois $\La$-order $F$:
\begin{enumerate}[{\rm (i)}]
\item $\mathrm{Irr}(F,\Fm)$ is finite;
\item $V^\Fm$ is finite-dimensional for every $[V]\in\mathrm{Irr}(F,\Fm)$.
\end{enumerate} 
The same holds with $F$ replaced by $F^{\op{op}}$.
\end{Theorem}

\begin{proof}
Immediate by Lemma \ref{lem:finite-fibers} and the main theorem of \cite{DroFutOvs1994}.
\end{proof}

The following is example where $\op{Irr}(F,\Fm)$ is infinite.

\begin{Example}\label{ex:infinite}
Let $\La=\C[t]$ and $\CH=\C[\partial]$ with $\partial$ primitive, acting as $d/dt$ on $\La$. Fix a nonzero polynomial $p(t)\in\La$. Let $F=F(p)$ be the subalgebra of $\C(t)\#\CH$ (in fact, of $\C[t]\#\CH$) generated by $\C[t]$ and one more element, $X=p(t)\partial$. $F$ is isomorphic to the algebra on two generators $t,X$ subject to $Xt-tX=p(t)$. Thus $F$ can also be realized  as an Ore extension $\C[t][X;\Id_{\C[t]},d/dt]$. Since $p$ is nonzero we have $\C(t)F=\C(t)\#\C[\partial]$, and since $p$ is a polynomial, $X$ preserves $\C[t]$. Hence $F(p)$ is a $\C[\partial]$-Galois $\C[t]$-order.
If $p(t)$ is a nonzero constant, then $F(p)$ is isomorphic to the first Weyl algebra.
Suppose $p(t)$ is not constant. Let $\la\in\C$ be a root of $p(t)$, and let $\mu\in\C$ be arbitrary. Define $V(\la,\mu)=\C v_{\la,\mu}$ with action $t v_{\la,\mu}=\la v_{\la,\mu}$ and $X v_{\la,\mu} = \mu v_{\la,\mu}$. Clearly the relation $Xt-tX=p(t)$ is preserved since $\la\mu-\mu\la=0=p(\la)$. Thus $V(\la,\mu)$ is a one-dimensional (hence simple and Harish-Chandra with respect to $\C[t]$) module over $F(p)$.
Let $\Fm_\la=(t-\la)$ be the maximal ideal of $\C[t]$ corresponding to $\la$.
Thus, when $\la$ is a root of $p(t)$, then $\op{Irr}_l\big(F(p),\Fm_\la\big)$ is uncountably infinite, as it contains the distinct isomorphism classes $\big[V(\la,\mu)\big]$ for $\mu\in\C$.
\end{Example}

\subsection{Sufficient condition for local finiteness}
\label{sec:loc-fin}

Let $\{\tilde\CH_r\}_{r=0}^\infty$ be the coradical filtration of $\tilde\CH$ and $\CH_r=\CH\cap\Delta^{-1}(\tilde\CH_r\otimes\CH)$ be the induced filtration on $\CH$. Put
\begin{equation}
F_r = F\cap (L\CH_r).
\end{equation}
Since each $\CH_r$ is a coideal of $\tilde\CH$ contained in $\CH$, each $F_r$ is a $(\La,\La)$-subbimodule of $F$. Furthermore $F_0\subset F_1\subset\cdots$ and $F=\cup_{r=0}^\infty F_r$.

\begin{Theorem}
\label{thm:finite-weight-spaces-general}
Let $F$ be an $\CH$-Galois $\La$-order, and $\Fm\in\cfs(\La)$.
Let $V$ be a left Harish-Chandra module over $F$ with respect to $\La$.
Suppose that there exists a non-negative integer $r\ge 0$ such that the following two conditions hold:
\begin{enumerate}[{\rm (i)}]
\item $V=F_r v_\Fm$ for some $v_\Fm\in V_\Fm$,
\item $\dim_\K \op{Stab}(\CH_r, \Fm)<\infty$.
\end{enumerate}
Then $\dim_\K V_\Fm<\infty$.
\end{Theorem}

\begin{proof}
By Lemma \ref{lem:finite-fibers-general},
$\displaystyle
\hat{(F_r)}_\Fm=\lim_{\longleftarrow} \frac{F_r}{\Fm^N F_r + F_r \Fm^N}
$
is finitely generated as a right $\hat\La_\Fm$-module.
Now $V_\Fm$ is a homomorphic image of $\hat{(F_r)}_\Fm\otimes \La/\Fm^n$ where $\Fm^n v_\Fm=0$, hence is finite-dimensional.
\end{proof}

\begin{Example}
Consider the rational-differential setting from Example \ref{ex:rational-differential}, $(\CH,\La)=\big(S(V)\rtimes W,\C[V]\big)$ where $V$ is a finite-dimensional complex vector space and $W\le\GL(V)$. Let $F$ be any $\CH$-Galois $\La$-order. It is natural to study $F$-modules $V$ generated by a weight vector $v_\la$ for $\C[V]$, where $\la\in V$ so that $p v_\la=p(\la)v_\la$ for all $p\in\C[V]$, and such that furthermore $F\cap \big(\C(V)S^{r+1}(V)\big)v_\la=0$ for large enough $r$. For such $r$ we have $F_r v_\la=V$ because here $F_r=F\cap \big(\C(V)\# S(V)_{(r)}\rtimes W\big)$. Thus condition (i) holds.
Furthermore $\op{Stab}(\CH_r,\Fm_\la)\cong S(V)_{(r)}\rtimes W_\la$ where $W_\la=\{w\in W\mid w(\Fm_\la)=\Fm_\la\}$. Thus condition (ii) of Theorem \ref{thm:finite-weight-spaces-general} holds iff $W_\la$ is finite.
\end{Example}

\section{Spherical Galois orders}
\label{sec:spherical}

\subsection{Setting and definition}
Let $(\CH,\La)$ be a setting. In this section we further assume that $\CH$ is of the form $\CH=\CH'\rtimes W$, where $\CH'$ is a $\tilde\CH$-coideal subalgebra of $\CH$, and $W$ is finite subgroup of the group of grouplike elements of $\CH$, whose order $|W|$ is invertible in $\K$, normalizing $\CH'$ and $(\K W)\cap\CH'=\K 1_\CH$. 
We assume that $(L\#\CH')^W$ acts faithfully on $L^W$.

The following is a generalization of the principal Galois orders from \cite{Har2020} (in the case when the ring $\La$ from \cite{Har2020} is assumed to be an algebra over a field $\K$ as in this paper).

\begin{Definition}
\label{def:spherical-Galois-order}
A \emph{spherical Galois order with respect to $(\CH',W,\La)$} is a subalgebra $U$ of $(L\#\CH')^W$ such that
\begin{enumerate}[{\rm (i)}]
\item $\La^W \subset U$;
\item $L^W U=(L\#\CH')^W$;
\item $\hat X(a)\in \La^W$ for all $X\in U$ and all $a\in\La^W$.
\end{enumerate}
\end{Definition}

\begin{Example}
The \emph{standard spherical Galois order} is defined as
\begin{equation}
\CK(\CH',W,\La) = \{X\in (L\#\CH')^W\mid \text{$\hat X(a)\in\La^W$ for all $a\in\La^W$}\}.
\end{equation}
The standard spherical Galois order is the unique maximal spherical Galois order with respect to $(\CH',W,\La)$.
\end{Example}

\begin{Example}
When $\CH$ is a group algebra and $\La$ is integrally closed, spherical Galois orders are the same thing as principal Galois orders \cite{Har2020}. Thus examples of spherical Galois orders include finite W-algebras of type $A$ \cite{FutMolOvs2010}\cite{Har2020} and $U_q(\Fgl_n)$ \cite{Har2020}, and Coulomb branches  \cite{Web2019}.
\end{Example}

\subsection{Spherical Galois orders as centralizer subalgebras}

Generalizing \cite{Web2019}, we show that the $W$-centralizer of an $\CH$-Galois $\La$-order is a spherical Galois order, and that any spherical Galois order occurs this way.

Put $\CF=L\#\CH'\rtimes W$ and $\CK=(L\#\CH')^W$.
Let $\CF_\La=\CF(\CH'\rtimes W, \La)$ be the standard $\CH$-Galois $\La$-order and $\CK_{\La^W}=\CK(\CH',W,\La)$ be the standard spherical Galois order with respect to $(\CH',W,\La)$.

\begin{Proposition}\label{prp:spherical}
Let $e=\frac{1}{|W|}\sum_{w\in W} w\in \K W$ be the symmetrizing idempotent.
\begin{enumerate}[{\rm (a)}]
\item The map
\begin{equation}
\psi:\CK \to e\CF e, \quad \psi(X)=eXe,
\end{equation}
is an algebra isomorphism;
\item $\psi(\CK_{\La^W})=e\CF_\La e$;
\item If $U$ is a spherical Galois order with respect to $(\CH',W,\La)$, then $\psi(U)=eFe$ for some $\CH$-Galois $\La$-order $F$ containing the standard $\K W$-Galois $\La$-order $\CF(\K W, \La)$;
\item If $F$ is an $\CH$-Galois $\La$-order, then there exists a spherical Galois order $U$ with respect to $(\CH',W,\La)$ such that
 $\psi(U)=eFe$.
\end{enumerate}
\end{Proposition}

\begin{proof}
Straightforward generalization of the proofs of \cite[Lem.~2.3, Lem.~2.5]{Web2019} but we provide some details for the convenience of the reader. 

(a) $\psi$ is an algebra map since $eX=Xe$ for $X\in\CK$. Recall that we assume the map $\rho:\CF\to \End(L)$, $X\mapsto\hat X$ is injective. Let $K=L^W$.
The restriction of $\rho$ to $\CK$ gives an injective map $\rho|_\CK:\CK\to\End(K)$. On the other hand if $X\in e\CF e$ then $\rho(X)$ also maps $K$ into $K$, giving an injective map $\rho|_{e\CF e}:e\CF e\to\End(K)$. We have $\rho\circ\psi=\rho|_{\CK}$. This proves that $\psi$ is injective. For any $X\in\CF$ we have $eXe=eX'e$ where $X'=\frac{1}{|W|}\sum_{w\in W} wXw^{-1}\in\CK$ is the symmetrization of $X$. Thus $\psi$ is surjective.

(b) If $X\in \CK_{\La^W}$, let $Y=\psi(X)$. Then $\hat Y (\La)\subset \hat e\hat X(\La^W)\subset \La^W\subset \La$. This shows $\subset$. Conversely, let $X\in \CF_\La$. Let $X'$ be as in the proof of part (a). Then $X'\in\CK_{\La^W}$ and $\psi(X')=eX'e=eXe$, proving $\supset$.

(c) Let $F$ be the subalgebra of $\CF$ generated by $eUe$ and $\CF(\K W,\La)$. By part (b) we have $e\CF(\K W,\La)e=e\La^W e$ and thus $F=\CF(\K W,\La)eUe\CF(\K W,\La)$ and $eFe=eUe=\psi(U)$. Since $\La\subset\CF(\K W,\La)$ we have $\La\subset F$. Since $\hat X(a)\subset\La$ for all $a\in\La$ when $X\in eUe \cup \CF(\K W,\La)$ the same is true for all $X\in F$. It remains to show that $LF=\CF$. Since $L\subset LF$ and $\K W\subset \CF(\K W,\La)\subset F$, it suffices to show that $\CH'\subset LF$. Since $1\in LeL$ (a well-known fact related to that $L\#\K W\cong \End_{L^W}(L)$ is Morita equivalent to $L^W$) we have
$\CH'\subset LeL\CH'LeL=Le(L\#\CH')eL=Le(L\#\CH')^WeL\subset LF$.

(d) Let $\si:\CF\to\CK$ be the symmetrization epimorphism 
given by $\si(X)=\frac{1}{|W|}\sum_{w\in W}wXw^{-1}$. Define $U$ to be the image $\si(F)$. 
We have $\La^W=\si(\La)\subset U$. If $X\in F$ and $\ga\in\Gamma$ then $\hat{\si(X)}(\ga)=\frac{1}{|W|}\sum_{w\in W}\hat w \hat X\hat w^{-1}(\ga)\in \frac{1}{|W|}\sum_{w\in W}\hat w\hat X(\La^W)\subset\frac{1}{|W|}\sum_{w\in W}\hat w(\La)\subset \La^W$.
Lastly, $L^WU=\si(L)\si(F)=\si(LF)=\si(\CF)=\CK$. Thus $U$ is a spherical Hopf Galois order. And $\psi(U)=eUe=e\si(F)e=eFe$ as required.
\end{proof}

Using this we can for example show that spherical Galois orders satisfy the Futorny-Ovsienko property:
\begin{Lemma}
Let $U$ be a spherical Galois order with respect to $(\CH',W,\La)$.
Let $V$ be any finite-dimensional left (right) $K$-subspace of $\CK$. Then $V\cap U$ is a finitely-generated left (right) $\Gamma$-module.
\end{Lemma}

\begin{proof}
Let $F$ be the corresponding $\CH$-Galois $\La$-order, by definition generated by $eUe$ and $\CF(\K W, \La)$. Then $LeVe$ is a finite-dimensional left $L$-subspace of $\CF$, hence by Theorem \ref{thm:FO}, $LeVe\cap F$ is finitely generated as a left $\La$-module, consequently also finitely generated as a left $\La^W$-module. Since $V\cap U\to LeVe\cap F$, $v\mapsto eve$, is injective, $V\cap U$ is finitely generated as a left $\Ga$-module. Similarly on the right.
\end{proof}

We obtain the following decomposition and maximal commutativity result, which can also be proved exactly as in Lemmas \ref{lem:direct-sum} and \ref{lem:maxcomm}.

\begin{Corollary}
Let $U$ be a spherical Galois order with respect to $(\CH',W,\La)$. Then
\begin{enumerate}[{\rm (a)}]
\item $U=\Ga\oplus U_-$, where $U_-=\big\{X\in U\mid \hat U(1_{\La^W})=0 \big\}$;
\item $\La^W$ is a maximal commutative subalgebra of $U$.
\end{enumerate}
\end{Corollary}

As in \cite[Lem.~2.8]{Web2019}, the quotient functor $M\mapsto eM$ from the category of left $F$-modules to the category of left $U$-modules restricts to a functor from the category of Harish-Chandra modules over $F$ with respect to $\La$ to the category of Harish-Chandra modules over $U$ with respect to $\La^W$. Combined with Theorem \ref{thm:main}, we obtain the following result, which generalizes one of the main result of \cite{FutOvs2014}.

\begin{Corollary}
Let $U$ be a spherical Galois order with respect to $(\CH',W,\La)$. Let $\Fm\in\cfs(\La)$ be such that the stabilizer $\op{Stab}(\CH,\Fm)$ is finite-dimensional. Let $\Fn=\La^W\cap\Fm$. Then there are only finitely many isomorphism classes of simple Harish-Chandra $U$-modules $V$ such that $V^\Fn\neq 0$. Furthermore, $V^\Fn$ is finite-dimensional for any simple Harish-Chandra $U$-module $V$.
\end{Corollary}

\subsection{Morita equivalence}

Following \cite{Web2019}, we say that an $\CH$-Galois $\La$-order $F$ is \emph{Morita for $W$} if $FeF=F$ (equivalently, if $1\in FeF$). Then $F$ is Morita equivalent to its corresponding spherical Galois order $U\cong eFe$. For example, if the $\K W$-Galois $\La$-order $F(\K W)=F\cap LW$ (see Example \ref{ex:sub-galois}) is Morita for $W$, then so is $F$. In particular this holds if $F(\K W)$ is isomorphic to $\End_{\La^W}(\La)$. This in turn holds if $F$ is obtained as in Proposition \ref{prp:spherical} for a spherical Galois order $U$ with respect to $(\CH',W,\La)$ where $\La=\C[V]$ and $W$ acts faithfully on $V$ as a complex reflection group.

\end{document}